\documentclass[a4paper, 11pt]{amsart}
\usepackage{mathptmx,amssymb,amscd,latexsym, eulervm}
\usepackage{amsmath}
\usepackage{amsthm}
\usepackage{mathdots}
\usepackage[colorlinks=true,citecolor=violet,linkcolor=blue,urlcolor=blue]{hyperref}
\usepackage[dvipsnames]{xcolor}
\usepackage[onehalfspacing]{setspace}
\usepackage{tabularx}
\usepackage{amsfonts}
\usepackage{paralist}
\usepackage{aliascnt}
\usepackage{amscd}
\usepackage{blkarray}
\usepackage{mathbbol}
\usepackage{setspace}
\usepackage{needspace}
\usepackage[inner=2.4cm,outer=2.4cm, bottom=3.2cm]{geometry}
\usepackage{tikz, tikz-cd}
\usepackage{calligra,mathrsfs}

\usepackage{tikz}
\usetikzlibrary{matrix}
\usetikzlibrary{arrows,calc}
\allowdisplaybreaks

\AtBeginDocument{%
	\def\MR#1{}
}

\makeatletter
\@namedef{subjclassname@2020}{%
	\textup{2020} Mathematics Subject Classification}
\makeatother

\newcommand{\sZ}{\mathcal{Z}}
\newcommand{\sU}{\mathscr{U}}

\newcommand{\kk}{\mathbb{k}}

\newcommand{\ZZ}{\mathbb{Z}}

\newcommand{\PP}{{\normalfont\mathbb{P}}}

\newcommand{\QQ}{\mathbb{Q}}

\newcommand{\II}{\mathscr{I}}

\newcommand{\sH}{\mathscr{H}}

\newcommand{\OO}{\mathscr{O}}

\newcommand{\HH}{{\normalfont\text{H}}}

\newcommand{\Hilb}{{\normalfont\text{Hilb}}}
\newcommand{\Spec}{\normalfont\text{Spec}}

\newcommand{\biProj}{{\normalfont\text{BiProj}}}

\def\f0{\mathbf{0}}

\def\1{\mathbf{1}}

\newtheorem{headthm}{Theorem}

\newaliascnt{headcor}{headthm}

\aliascntresetthe{headcor}

\newaliascnt{headconj}{headthm}

\aliascntresetthe{headconj}

\newaliascnt{corollary}{theorem}
\newtheorem{corollary}[corollary]{Corollary}
\aliascntresetthe{corollary}

\newaliascnt{claim}{theorem}

\aliascntresetthe{claim}

\newaliascnt{lemma}{theorem}

\aliascntresetthe{lemma}

\newaliascnt{conjecture}{theorem}

\aliascntresetthe{conjecture}

\newaliascnt{proposition}{theorem}
\newtheorem{proposition}[proposition]{Proposition}
\aliascntresetthe{proposition}

\theoremstyle{definition}
\newaliascnt{definition}{theorem}

\aliascntresetthe{definition}

\newaliascnt{notation}{theorem}

\aliascntresetthe{notation}

\newaliascnt{example}{theorem}
\newtheorem{example}[example]{Example}
\aliascntresetthe{example}

\newaliascnt{examples}{theorem}

\aliascntresetthe{examples}

\newaliascnt{remark}{theorem}
\newtheorem{remark}[remark]{Remark}
\aliascntresetthe{remark}

\newaliascnt{question}{theorem}

\aliascntresetthe{question}

\newaliascnt{questions}{theorem}

\aliascntresetthe{questions}

\newaliascnt{problem}{theorem}

\aliascntresetthe{problem}

\newaliascnt{construction}{theorem}

\aliascntresetthe{construction}

\newaliascnt{setup}{theorem}

\aliascntresetthe{setup}

\newaliascnt{algorithm}{theorem}

\aliascntresetthe{algorithm}

\newaliascnt{observation}{theorem}

\aliascntresetthe{observation}

\newaliascnt{defprop}{theorem}

\aliascntresetthe{defprop}

\newaliascnt{fact}{theorem}

\aliascntresetthe{fact}

\DeclareFontFamily{OT1}{pzc}{}
\DeclareFontShape{OT1}{pzc}{m}{it}{<-> s * [1.100] pzcmi7t}{}
\DeclareMathAlphabet{\mathchanc}{OT1}{pzc}{m}{it}

\DeclareMathOperator{\sHom}{\mathchanc{Hom}}

\def\equationautorefname~#1\null{(#1)\null}
\def\sectionautorefname~#1\null{Section #1\null}
\def\subsectionautorefname~#1\null{\S #1\null}

\def\surjects{\twoheadrightarrow}

\title{Disconnected multigraded Hilbert schemes on $\PP^2\times\PP^1$}
\author{Yairon Cid-Ruiz}
\address{Department of Mathematics, North Carolina State University, Raleigh, NC 27695, USA}
\email{ycidrui@ncsu.edu}

\date{\today}

\begin{document}

\begin{abstract}
We exhibit an infinite family of disconnected multigraded Hilbert schemes on the biprojective space $\PP^2_{\kk}\times_\kk\PP^1_{\kk}$.  
More precisely, for any integer $a \ge 2$ and $p_a(z_1,z_2) = 2az_1+az_2+3a-2a^2 \in \QQ[z_1,z_2]$, the multigraded Hilbert scheme $\Hilb_{p_a}(\PP^2_{\kk}\times_\kk\PP^1_{\kk})$ is disconnected.  
As a consequence, there exist infinitely many disconnected Haiman--Sturmfels multigraded Hilbert schemes even for a standard bigrading on a polynomial ring in only five variables.
\end{abstract}

\maketitle


\section{Introduction}

For every numerical polynomial $p(z) \in \QQ[z]$, a fundamental result of Hartshorne \cite{HartshorneConnectedness} shows that the
Grothendieck Hilbert scheme $\Hilb_{p}(\PP_\kk^n)$ is always connected (see also \cite{Reeves, Pardue, PeevaStillmanConnected} for subsequent proofs).
By contrast, there are well-known examples of disconnected Hilbert schemes with a fixed multigraded Hilbert function.  
Haiman and Sturmfels \cite{HaimanSturmfels} developed the general multigraded Hilbert scheme, Peeva and Stillman \cite{PeevaStillmanToric} studied the toric case, and Santos \cite{Santos} constructed disconnected toric Hilbert schemes.  
Those examples concern fixed Hilbert functions for affine semigroup gradings.
They do not decide whether disconnectedness can occur for the multigraded Hilbert scheme $\Hilb_p(\PP)$ of a product of projective spaces $\PP= \PP_\kk^{n_1} \times_\kk \PP_\kk^{n_2} \times_\kk \cdots \times_\kk \PP_\kk^{n_m}$ with multigraded Hilbert polynomial $p=p(z_1,\ldots,z_m) \in \QQ[z_1,\ldots,z_m]$.  
Since $\PP$ is a smooth projective toric variety, the existence and projectivity of $\Hilb_p(\PP)$ follow from work of Maclagan and Smith \cite{MaclaganSmith}.
Recent work has investigated the geometry and connectedness of several classes of multigraded Hilbert schemes (see \cite{CartwrightSturmfelsDiagonal,MaclaganSmithSmooth,HeringMaclaganTGraph,AholtSturmfelsThomas,MandziukLimitsToric,JelisiejewOpen,RamkumarSammartano,Ablett}).
Nevertheless, to the best of our knowledge, the connectedness of $\Hilb_p(\PP)$ has remained open in general.
We answer this question negatively by exhibiting the following infinite family of counterexamples on the biprojective space $\PP_\kk^2 \times_\kk \PP_\kk^1$.

\begin{headthm}\label{thm:main}
Let $\kk$ be a field, $a\geq2$ be an integer, and  $p_a(z_1,z_2)=2az_1+az_2+3a-2a^2\in\QQ[z_1,z_2]$.  
Then the multigraded Hilbert scheme $\Hilb_{p_a}(\PP^2_{\kk}\times_\kk\PP^1_{\kk})$ is disconnected.
\end{headthm}

For instance, when $a=2$, we obtain that the multigraded Hilbert scheme $\Hilb_{4z_1+2z_2-2}(\PP^2_{\kk}\times_\kk\PP^1_{\kk})$ is disconnected. 
As a direct consequence, there exist infinitely many disconnected Haiman--Sturmfels multigraded Hilbert schemes for a standard bigrading on a polynomial ring in five variables (see \autoref{cor:HS-disconnected}).

\section*{Acknowledgments}

The author received support from NSF grant DMS-2502321 and Simons Foundation Travel Support for Mathematicians Award MPS-TSM-00013551.
We are grateful to  Ritvik Ramkumar and Alessio Sammartano for several helpful conversations.

\noindent
\textbf{AI disclosure.}
OpenAI Codex, with GPT-5.5 and GPT-5.6,  was used for exploratory computations that helped identify the counterexample for $a=2$. 
The author conceived the theoretical framework, wrote the manuscript, and assumes full responsibility for its contents.

\section{Proofs of our results}

Let $\kk$ be a field and $S := \kk[x_0,x_1,x_2,y_0,y_1]$ be a standard bigraded polynomial ring with $\deg(x_i) = (1,0)$ and $\deg(y_j) = (0,1)$.
Fix an integer $a\geq2$.
We set 
$$
\PP \;:=\; \PP_\kk^2 \times_\kk \PP_\kk^1 = \biProj(S), \qquad p_a(z_1,z_2) \;:=\;2az_1+az_2+3a-2a^2 \qquad  \text{and} \qquad \sH_a \;:=\; \Hilb_{p_a}(\PP).
$$
The chosen polynomial is precisely the multigraded Hilbert polynomial of the complete intersections considered below.  

\begin{remark}
	Let	$Y=V(\ell,g)\subset \PP$ where $\deg(\ell)=(1,0)$, $\deg(g)=(a,2a)$, and $\ell,g$ form a regular sequence. 
	Let $Q=V(\ell)\cong\PP^1_\kk\times_\kk\PP^1_\kk \subset \PP$.  
	Then $Y$ is an effective Cartier divisor of type $(a,2a)$ on $Q$.  
	A direct computation gives
	$
	\chi\bigl(Y,\mathcal \OO_Y(z_1,z_2)\bigr)
		=(z_1+1)(z_2+1)-(z_1-a+1)(z_2-2a+1)
		=p_a(z_1,z_2).
	$
	Thus every complete intersection of bidegrees $(1,0)$ and $(a,2a)$
	defines a point of $\sH_a$.
\end{remark}

The key point of our construction is that this complete intersection locus has two opposite cohomological semicontinuity descriptions.
Our motivation to use these cohomological descriptions comes from the \emph{fiber-full scheme} \cite{CidRuizRamkumarFiber, CidRuizRamkumarLocal} (see \autoref{rem:fiber-full} and \autoref{rem:two-points}).

\begin{proposition}\label{prop:cohomological-characterization}
Let $Y \subset \PP = \PP_\kk^2 \times_\kk \PP_\kk^1$ be a closed subscheme such that $\chi\left(Y, \OO_Y(z_1,z_2)\right)=p_a(z_1,z_2)$.  
Then the following conditions are equivalent:
\begin{enumerate}[\rm (i)]
 \item $\HH^0\bigl(\PP,\mathcal \II_Y(1,0)\bigr)\neq0$.
 \item $h^0\bigl(Y,\mathcal \OO_Y(1,0)\bigr)\leq2$.
 \item $h^0\bigl(Y,\mathcal \OO_Y(1,0)\bigr)=2$.
 \item $Y = V(\ell, g) \subset \PP$ with  $\deg(\ell) = (1,0)$ and $\deg(g) = (a,2a)$.
\end{enumerate}
\end{proposition}
\begin{proof}
		Note that the condition (iv) implies that $Y$ is a complete intersection since by assumption  $\dim(Y)=1$ and $\dim(\PP) = 3$.
		Throughout this proof we use the K\"unneth formula \cite[\href{https://stacks.math.columbia.edu/tag/0BEC}{Tag 0BEC}]{stacks-project}  several times.
		
	(iii) $\Rightarrow$ (ii): This implication is clear.

	(ii) $\Rightarrow$ (i):
If (ii) holds, then the exact sequence
$$
0 \;\longrightarrow\; \HH^0\bigl(\PP,\II_Y(1,0)\bigr) \;\longrightarrow\; \HH^0\bigl(\PP,\mathcal \OO_{\PP}(1,0)\bigr)
 \;\longrightarrow\; \HH^0\bigl(Y,\mathcal \OO_Y(1,0)\bigr)
$$
implies that $\HH^0\bigl(\PP,\II_Y(1,0)\bigr)\neq0$ because $h^0\bigl(\PP,\mathcal \OO_{\PP}(1,0)\bigr)=3$.
This proves the implication (ii) $\Rightarrow$ (i).

(i) $\Rightarrow$ (iv):
Assume that (i) holds and choose a nonzero element
$\ell\in \HH^0(\PP,\II_Y(1,0))$.
Let $Q=V(\ell) \cong\PP_\kk^1\times_\kk\PP_\kk^1 \subset \PP$.
Let $\II = \II_{Y/Q} \subset \OO_Q$ be the ideal sheaf of $Y$ in $Q$.
Double-dualizing
$\II\hookrightarrow\OO_Q$ gives the generically injective map 
$$
\II^{\vee\vee} = \sHom_{\OO_Q}\left(\sHom_{\OO_Q}(\II, \OO_Q), \OO_Q\right) \;\longrightarrow\; \OO_Q;
$$ 
this map is actually injective because $\II^{\vee\vee}$ is torsion-free. 
Since $Q$ is smooth, the rank-one
reflexive sheaf $\II^{\vee\vee}$ is a line bundle.\footnote{For more details on reflexive sheaves and related topics, see \cite[\href{https://stacks.math.columbia.edu/tag/0AVT}{Tag 0AVT}]{stacks-project}, \cite[\S 1.4]{BH}.} 
Therefore, there is an effective Cartier divisor $D$ on $Q$ such that  
$$
\II^{\vee\vee} \;\cong\; \mathcal \OO_Q(-D) \;\cong\; \OO_Q(-d,-e).
$$ 
Furthermore, since $\II$ and $\II^{\vee\vee}\cong \OO_Q(-D)$ agree in codimension one, the quotient
$\mathcal K=\mathcal \OO_Q(-D)/\II$ is zero-dimensional.
Let $r = \chi(Q, \mathcal{K}(z_1,z_2)) = h^0(Q, \mathcal{K})$.
From the short exact sequence $0\rightarrow\mathcal \OO_Q(-d,-e)\rightarrow\mathcal \OO_Q\rightarrow\mathcal \OO_D\rightarrow0$, we obtain
\begin{align*}
 \chi\bigl(D,\mathcal \OO_D(z_1,z_2)\bigr)
 &\;=\;(z_1+1)(z_2+1)-(z_1-d+1)(z_2-e+1)\\
 &\;=\;(z_1+1)(z_2+1)-\bigl((z_1+1)(z_2+1)-e(z_1+1)-d(z_2+1)+de\bigr)\\
 &\;=\;ez_1+dz_2+d+e-de.
\end{align*}
Then the short exact sequence $ 0\rightarrow\mathcal K\rightarrow \OO_Y \rightarrow \OO_D\rightarrow 0$ yields
$$
 p_a(z_1,z_2) \;=\; \chi\left(Y, \OO_Y(z_1,z_2)\right) \;=\;
 ez_1+dz_2+d+e-de+r.
$$
By assumption $p_a(z_1,z_2) = 2az_1+az_2+3a-2a^2$, and so we get $d = a$, $e= 2a$ and $r=0$.
Therefore, $Y=D$ is an effective Cartier divisor of type $(a,2a)$ on $Q \cong \PP_\kk^1 \times_\kk \PP_\kk^1$. 
 Its defining section lifts to $\PP$ because
$\HH^1(\PP,\mathcal \OO_{\PP}(a-1,2a))=0$ and so the restriction map $\HH^0(\PP,\mathcal \OO_{\PP}(a,2a)) \surjects \HH^0(Q,\mathcal \OO_Q(a,2a))$ is surjective. 
This proves (i) $\Rightarrow$ (iv).

(iv) $\Rightarrow$ (iii):
Assume that (iv) holds and write $Y = V(\ell, g) \subset \PP$ with $\deg(\ell)=(1,0)$ and $\deg(g)=(a, 2a)$.
Let $Q = V(\ell) \cong \PP_\kk^1 \times_\kk \PP_\kk^1 \subset \PP$.
By combining the short exact sequence
$$
 0\longrightarrow\mathcal \OO_Q(1-a,-2a)
 \longrightarrow\mathcal \OO_Q(1,0)
 \longrightarrow\mathcal \OO_Y(1,0)\longrightarrow0
$$
with the vanishings
$\HH^0(Q,\mathcal \OO_Q(1-a,-2a))
 =0$ and $\HH^1(Q,\mathcal \OO_Q(1-a,-2a))=0$,\footnote{Here we use the condition $a\ge2$.}
we obtain the equalities
$
 h^0(Y,\mathcal \OO_Y(1,0))=h^0(Q,\mathcal \OO_Q(1,0))=2.
$
This concludes the proof of the proposition.
\end{proof}

\begin{remark}\label{rem:fiber-full}
To simplify the notation, suppose in this remark that $a=2$.
Fix positive integers $u,v>0$. 
Let $Y=V(\ell,g)\subset \PP$ be as in condition~(iv) of
\autoref{prop:cohomological-characterization}.  
The line bundle $\OO_{\PP}(u,v)$ defines a Segre--Veronese embedding of $\PP$.
Let $Q=V(\ell) \cong \PP_\kk^1 \times_\kk \PP_\kk^1 \subset \PP$.
For every $t\in\ZZ$, set $h_Y^i(t):=h^i(Y,\OO_Y(ut,vt))$.
From the short exact sequence $0\rightarrow \OO_Q(ut-2,vt-4) \rightarrow\OO_Q(ut,vt) \rightarrow \OO_Y(ut,vt)\rightarrow0$, we can compute that
$$
h_Y^0(t) \;=\;
\begin{cases}
0 & \text{if }t<0,\\
1 & \text{if }t=0,\\
(4u+2v)t-2 & \text{if }t>0,
\end{cases}
\qquad
h_Y^1(t) \;=\;
\begin{cases}
-(4u+2v)t+2 & \text{if }t<0,\\
3 & \text{if }t=0,\\
0 & \text{if }t>0,
\end{cases}
$$
and $h_Y^i(t)=0$ for $i\geq2$.
Thus this cohomology table is independent of the specific $Y$.  
We obtain that the Segre--Veronese embeddings of all of these $Y$ belong to the same fiber-full scheme (see \cite{CidRuizRamkumarFiber,CidRuizRamkumarLocal}).
\end{remark}

For every (not necessarily closed) point $\tau \in \sH_a$, let $\kappa(\tau)$ be the residue field of the local ring $\OO_{\sH_a, \tau}$ and 
$$
Y_\tau \;\subset\; \PP_\tau \;:=\; \PP \times_\kk \kappa(\tau) \;\cong\; \PP_{\kappa(\tau)}^2 \times_{\kappa(\tau)} \PP_{\kappa(\tau)}^1
$$ 
be the corresponding closed subscheme with Hilbert polynomial $p_a(z_1,z_2)$.
The next corollary shows that the locus studied in \autoref{prop:cohomological-characterization} is both open and closed in $\sH_a=\Hilb_{p_a}(\PP_\kk^2 \times_\kk \PP_\kk^1)$.

\begin{corollary}\label{cor:clopen}
The following subset of $\sH_a$ is open and closed:
$$
\sZ_a \;:=\; \Big\lbrace \tau \in \sH_a \;\mid\;  \text{$Y_\tau \subset \PP_\tau$ is a complete intersection of bidegrees $(1,0)$ and $(a,2a)$} \Big\rbrace.
$$ 
\end{corollary}
\begin{proof}
Let $\sU\subset \mathfrak{X} := \PP\times_\kk \sH_a$ be the universal family and $\II_\sU \subset \OO_{\mathfrak{X}}$ be its ideal sheaf.  
Let 
$$
\OO_\sU(1,0) \;:=\; \OO_\sU \otimes \pi^*(\OO_{\PP}(1,0)) \quad  \text{ and } \quad  \II_\sU(1,0) \;:=\; \II_\sU \otimes \pi^*(\OO_{\PP}(1,0)),
$$ 
where $\pi : \mathfrak{X} \rightarrow \PP$ is the natural projection.
Let $q : \mathfrak{X} \rightarrow \sH_a$ be the natural projection and $\mathfrak{X}_\tau = q^{-1}(\tau)= \mathfrak{X} \times_{\sH_a} \Spec(\kappa(\tau))$ be the corresponding fiber for any $\tau \in \sH_a$.
Both $\OO_\sU(1,0)$ and $\II_\sU(1,0)$ are flat over $\sH_a$.
Then the upper semicontinuity theorem (see, e.g., \cite[\S III.12]{HARTSHORNE}) shows that
$$
 \Bigl\{\tau \in \sH_a \;\mid\; h^0\big(\PP_\tau,\II_{Y_\tau}(1,0)\big) = h^0\big(\mathfrak{X}_\tau, \II_\sU(1,0) \otimes \kappa(\tau)\big) \geq1\Bigr\}
$$
is closed, whereas
$$
 \Bigl\{\tau \in \sH_a \;\mid\; h^0\big(Y_\tau,\OO_{Y_\tau}(1,0)\big) = h^0\big(\mathfrak{X}_\tau, \OO_\sU(1,0) \otimes \kappa(\tau)\big) \leq2\Bigr\}
$$
is open. 
By \autoref{prop:cohomological-characterization}, 
both subsets coincide with $\sZ_a$.
\end{proof}

\begin{example}\label{ex:inside}
		The complete intersection
		$
		Y_{1,a} = V\left(x_0,x_1^ay_0^{2a}\right) \subset \PP
		$
		yields a distinguished point inside the subset $\sZ_a \subset \sH_a$.
\end{example}

We are ready to prove our main result. 

\begin{proof}[Proof of \autoref{thm:main}]
		By \autoref{cor:clopen} and \autoref{ex:inside}, we have that $\sZ_a$ is a nonempty open and closed subset of $\sH_a$.
		Let $I_{2,a} = (x_0^{2a},x_0y_0,x_1^ay_0)$, $J_a=(x_0,x_1^a)$ and $K_a = (x_0^{2a},y_0)$.
		Notice that $I_{2,a} = J_a \cap K_a$.
		Consider the closed subscheme $Y_{2,a} = V(I_{2,a}) \subset \PP$. 
		Since $\HH^0\left(\PP, \II_{Y_{2,a}}(1,0)\right)=0$, \autoref{prop:cohomological-characterization} implies that $[Y_{2,a}] \notin \sZ_a$.
		Thus to complete the proof, it suffices to show that $[Y_{2,a}] \in \sH_a$.
		From the short exact sequence 
		$$
		 0 \;\longrightarrow\; \OO_{Y_{2,a}}
		 \;\longrightarrow\;  \OO_{V(J_a)} \oplus \OO_{V(K_a)}
		 \;\longrightarrow\; \OO_{V(J_a+K_a)} \;\longrightarrow\; 0,
		$$
		we obtain
		\begin{align*}
			\chi\left(Y_{2,a}, \OO_{Y_{2,a}}(z_1,z_2)\right) &\;=\; \chi\left(\PP, \OO_{V(J_a)}(z_1,z_2)\right) + \chi\left(\PP, \OO_{V(K_a)}(z_1,z_2)\right) -\chi\left(\PP, \OO_{V(J_a+K_a)}(z_1,z_2)\right) \\
			&\;=\; a(z_2+1) \;+\; (2az_1+3a-2a^2) \;-\; a 
			\;=\; p_a(z_1,z_2).
		\end{align*}
		Therefore $[Y_{2,a}]\in\sH_a$.  
		This shows that $\sH_a$ is disconnected and concludes the proof of the theorem.
\end{proof}

\begin{remark}
		The openness of $\sZ_a$ is consistent with the classical projective setting: see \cite[Proposition~2.2.1]{BenoistThesis}, \cite[Corollary~6.20]{BertoneCioffiOrthSeiler}. 
		The unexpected feature of the multiprojective setting is that, for every $a\geq2$, the complete intersection locus $\sZ_a$ is also a proper closed subset.
\end{remark}

Although our proof is not combinatorial and instead relies on controlling certain cohomology loci, we identify the following two distinguished points of $\sH_a$, which \emph{cannot} be connected by a sequence of flat degenerations.

\begin{remark}
	\label{rem:two-points}
	With respect to the degree-lexicographic order induced by
	$x_0>x_1>x_2>y_0>y_1$, the two defining ideals
	\[
	I_{1,a}=I_{Y_{1,a}} \;=\; \left(x_0,x_1^ay_0^{2a}\right)
	\qquad\text{and}\qquad
	I_{2,a}=I_{Y_{2,a}} \;=\; \left(x_0^{2a},x_0y_0,x_1^ay_0\right)
	\]
	are \emph{bilex ideals} in the sense of \cite{AramovaCronaDeNegri}.
	These two bilex points lie in distinct connected components of $\sH_a = \Hilb_{p_a}(\PP_\kk^2 \times_\kk \PP_\kk^1)$.
	Moreover, with respect to the Segre embedding given by $\OO_{\PP}(1,1)$, we have $h^0(Y_{1,a}, \OO_{Y_{1,a}}(1,1))=4$ and $h^0(Y_{2,a}, \OO_{Y_{2,a}}(1,1))=a+3$; thus the Segre embeddings of $Y_{1,a}$ and $Y_{2,a}$ do not belong to the same fiber-full scheme.
\end{remark}

\begin{corollary}\label{cor:HS-disconnected}
For every integer $a\geq2$, there exists a function $h_a:\mathbb Z^2\to\mathbb N$ such that the
Haiman--Sturmfels multigraded Hilbert scheme
$\operatorname{Hilb}_S^{h_a}$ is disconnected. The functions $h_a$ are pairwise distinct.
\end{corollary}
\begin{proof}
By \cite[Theorem 6.2 and its proof]{MaclaganSmith}, we can identify $\sH_a$ with $\operatorname{Hilb}_S^{h_a}$ for a function $h_a$ with eventual polynomial $p_a$; the $h_a$ are pairwise distinct because the $p_a$ are.
\end{proof}

\bibliographystyle{amsalpha}
\bibliography{references}

\end{document}